\documentclass[a4paper,11pt]{amsart}
\def\Vo{\mathop{\mathrm{Vo}}}
\def\Vol{\mathop{\mathrm{Vol}}}

\def\ge{\geqslant}
\def\geq{\geqslant}
\def\le{\leqslant}
\def\leq{\leqslant}

\newtheorem{theo}{Theorem}[section]
\newtheorem{coro}[theo]{Corollary}

\newtheorem{prop}[theo]{Proposition}

\newtheorem{cons}[theo]{Construction}
\theoremstyle{definition}
\newtheorem{defi}[theo]{Definition}
\newtheorem{rema}[theo]{Remark}
\numberwithin{equation}{section}
\newtheorem{exam}{Example}

\usepackage{graphics}
\usepackage[cp1251]{inputenc}
\usepackage[T2A]{fontenc}
\usepackage{tikz-cd}
\usepackage{yfonts}
\usepackage{amsfonts,amssymb,amsmath,mathrsfs,amsthm,amscd}

\begin{document}

\title[Minimal and H-minimal submanifolds in toric geometry]{Minimal and H-minimal submanifolds in toric geometry}

\author{Artem Kotelskiy}

\begin{abstract}
In this paper we investigate a family of Hamiltonian-minimal Lagrangian submanifolds in ${\mathbb C}^m$, ${\mathbb C}P^m$ and other symplectic toric manifolds constructed from intersections of real quadrics. In particular we explain the nature of this phenomenon by proving H-minimality in a more conceptual way, and prove minimality of the same submanifolds in the corresponding moment-angle manifolds.
\end{abstract}

\maketitle

\tableofcontents

\section{Introduction}
Hamiltonian minimality (H-minimality for short) for Lagrangian submanifolds is a symplectic analogue
of Riemannian minimality. A Lagrangian embedding is called H-minimal if the variations of its volume
along all Hamiltonian vector fields are zero. This notion was introduced in the work of Y.-G. Oh~\cite{O}
in connection with the celebrated Arnold conjecture on the number of fixed points of a Hamiltonian
symplectomorphism. A simple example of H-minimal Lagrangian submanifold is a coordinate torus
[9] $S^1_{r_1}\times  \dots \times S^1_{r_m} \subset {\mathbb C}^m$, where $S^1_{r_k}$ is a circle of radius
$r_k > 0$ in $k$-th coordinate subspace of ${\mathbb C}^m$. Other examples of H-minimal Lagrangian submanifolds in
${\mathbb C}^m$ were constructed in~\cite{AC},~\cite{CU},~\cite{HR}. A lot of examples and
current problems can be found in the survey ~\cite{MO}.

In 2003 A.Mironov~\cite{M} suggested a universal construction of H-minimal Lagrangian embeddings $N \hookrightarrow {\mathbb C}^m$ based on
intersections of real quadrics $Z$ of special type. The same intersections of quadrics appear in toric topology as (real) moment-angle manifolds~\cite[Section 6.1]{BP}. Using the methods developed by Y.Dong~\cite{D} and Hsiang-Lawson~\cite{HL} we prove minimality of embeddings $N \hookrightarrow Z$, and explain the underlying reasons for H-minimality for embeddings $N \hookrightarrow {\mathbb C}^m$. Also, based on the note of A.Mironov and T.Panov~\cite{MP2} where they refer to the results of Y.Dong, we give an explicit construction and independent proof of H-minimality of Lagrangian submanifolds in symplectic toric manifolds, which generalize the result of A.Mironov.

The paper is organized as follows: after introduction we give some background material on the notions of minimality and H-minimality. In the third section we review the construction of symplectic toric manifolds. The main results are stated in the last two sections.

\section{Minimality and H-minimality}

\subsection{Minimality}

Let $M$ and $L$ be smooth compact manifolds, $g$ be a Riemannian metric on $M$. Assume that $i:L \hookrightarrow M$ is an embedding, i.e. $L$ is a submanifold in $M$. We endow $L$ with the metric induced from $M$.

\begin{defi}
A smooth {\itshape variation} of $i \colon L \to M$ is $C^{\infty}$-map $i:[-\epsilon,+\epsilon]\times L \rightarrow M$ such that all maps $i_t = i(t,\cdot):L \hookrightarrow M$ are embeddings and $i_0 =i$.
\end{defi}

Denote $i_t(y)=y_t$, $ \frac d {dt} i_t(y)= X(y_t)$, $i_t(L)=L_t$. In this notation, we say that variation $i_t$ happens along the vector field $X$.
\begin{defi}
An embedding $i: L \hookrightarrow M$ is called {\itshape minimal} if volume of $L$ is stationary with respect to all variations, i.e.
\begin{equation} 
\left. \frac d{dt} \right |_{t=0} \Vol(L_t) = 0.
\end{equation}
\end{defi}

Here we formulate a classical criterion of minimality. The proof can be found in~\cite[Theorem~4]{L}.
\begin{theo}[First variation formula]
\begin{equation}
\left. \frac d{dt} \right |_{t=0} \Vol(L_t) = - \int_L \langle H, X \rangle, 
\end{equation}
where $H$ is the mean curvature vector field of embedding $i: L \hookrightarrow M$ and $X$ is the variation vector field.
\end{theo}
\begin{coro}
An embedding $i:L \hookrightarrow M$ is minimal if and only if $H \equiv 0$.
\end{coro}

Let us formulate and prove an important criterion for minimality of $G$-invariant submanifolds, due to Hsiang-Lawson~\cite{HL}. Let $G$ be a compact connected Lie group, which acts on manifold $M$ by isometries. An embedding $i: L \hookrightarrow M$ is called $G$-invariant if there is a smooth action $G:L$ such that $ig=gi$ for all $g \in G$. A variation $i_t$ of a $G$-invariant embedding is equivariant if $i_t g= g i_t$ for all $g \in G$ and $t \in [-\epsilon,+\epsilon]$.
\begin{theo}\label{theo2.3}
Let $i: L \hookrightarrow M$ be a $G$-invariant embedding. Then embedding $i: L \hookrightarrow M$ is minimal if and only if volume of $L$ is stationary with respect to all equivariant variations.
\end{theo}
\begin{proof}
Assume that the volume of $L$ is stationary with respect to all equivariant variations.

Let $H$ be the mean curvature vector field on $i(L)$. Because $H$ depends only on $i$, which is $G$-invariant, we have $g_{*}H=H$ for all $g \in G$.
Assume that $\phi$ is a smooth $G$-invariant function on $L$. Define the following variation $i_t$, $-\epsilon < t < \epsilon$:
\begin{equation}
i_t(y)=y_t=exp_{y_0}[t\phi(y_0)H(y_0)].
\end{equation}
We chose $\epsilon > 0$ small enough so that all $i_t$ are embeddings. Notice that for all $g \in G$,
\begin{align}
\begin{split}
& g \circ i_t(y) = g \circ exp_{y_0}[t\phi(y_0)H(y_0)] =  exp_{g y_0}[(g_* (t\phi(y_0)H(y_0))] =          \\
&=exp_{g y_0}[t\phi( y_0)g_* H( y_0)] =  exp_{g y_0}[t\phi(g y_0) H( g y_0)] = i_t \circ g(y),
\end{split}
\end{align}
because $g y_0 = g i(y) = i (gy)$, and $\phi$ is a $G$-invariant function. Hence $i_t$ is an equivariant variation.

Notice that the variation $i_t$ happens along $\phi H$ (by definition), and therefore by the first variation formula we get
\begin{equation} 
\left. \frac d{dt} \right |_{t=0} \Vol(L_t) = - \int_L \phi |H|^2 , 
\end{equation}
which is zero because the variation is equivariant. This, together with the arbitrariness of $\phi$, implies $H \equiv 0$ which finishes the proof.
\end{proof}

\subsection{H-minimality}

Let $M$ be a Kaehler manifold with symplectic structure $\omega$, almost complex structure $J$ and metric $g$. A vector field $X$ is called Hamiltonian if $i_X \omega = \omega (X,\cdot)=df$, where $f$ is a smooth function on $M$.
\begin{defi}
A Lagrangian embedding $i:L \hookrightarrow M$ is called {\itshape Hamiltonian-minimal (H-minimal)} if volume of $L$ is stationary with respect to variations along all Hamiltonian vector fields.
\end{defi}
\begin{prop}[First variation formula]
Let $(M,\omega,g)$ be a Kaehler manifold. A Lagrangian submanifold $L \subset M$ is H-Minimal if and only if its mean curvature vector H satisfies
\begin{equation}
\delta i_H \omega = 0 
\end{equation}
on $L$, where $\delta$ is the Hodge dual operator of $d$ on $L$.
\end{prop}
\begin{proof}
By the first variation formula along a Hamiltonian vector field X,
\begin{multline} \left. \frac d {dt} \right |_{t=0}  \Vol(L_t) = - \int_L \langle H , X \rangle = - \int_L \langle i_H \omega , i_X \omega \rangle \\ = - \int_L \langle i_H \omega , d f \rangle  = - \int_L \langle \delta i_H \omega, f \rangle ,
\end{multline}
where $L_t$ is the deformation along X. Since $f$ is arbitrary, $\left. \frac d {dt} \right |_{t=0}   \Vol(L_t) = 0$ if and only if $ \delta i_H \omega = 0 $, as needed.
\end{proof}

There is the following criterion due to Y.Dong~\cite{D}, for H-minimality of $G$-invariant Lagrangian embeddings, similar to Theorem~\ref{theo2.3}:
\begin{prop}
Let $G:M$ be a symplectic action by isometries, where $G$ is a compact connected Lie group. Let $i: L \hookrightarrow M$ be a $G$-invariant Lagrangian embedding. Then $i: L \hookrightarrow M$ is H-minimal if and only if volume of $L$ is stationary with respect to all equivariant variations along Hamiltonian vector fields.
\end{prop}
\begin{proof}

Assume that volume of $L$ is stationary with respect to all equivariant variations along Hamiltonian vector fields.

Let $H$ be the mean curvature vector field on $i(L)$. Because vector field $H$ depends only on embedding $i$, which is $G$-invariant, we have $g_{*}H=H$ for all $g \in G$. Since the action is symplectic, both $i_H\omega$ and $\delta i_H\omega$ are $G$-invariant. Assume that $\phi$ is a smooth $G$-invariant function on $L$. Define the following variation $i_t$, $-\epsilon < t < \epsilon$:
\begin{equation} 
i_t(y)=y_t=exp_{y_0}[tV],
\end{equation}
where $V$ is defined by the condition $JV=\bigtriangledown (\phi \delta i_H\omega)$, which is equivalent to the condition $i_V \omega =  d(\phi \delta i_H \omega)$. We choose $\epsilon > 0$ small enough so that all $i_t$ are embeddings. Notice that for all $g \in G$,
\begin{align}
\begin{split}
& g \circ i_t(y) = g \circ exp_{y_0}[t V(y_0)]
=exp_{g y_0}[t g_* V(y_0)] \\
&=  exp_{g y_0}[t V(g y_0)] = i_t \circ g(y),
\end{split}
\end{align}
because $g y_0 = g i(y) = i (gy)$ and $V$ is $G$-invariant. Hence $i_t$ is an equivariant variation.

Since the variation $i_t$ happens along the vector field $V$, the first variation formula implies
\begin{equation}
    \left. \frac {{d}}{{d}t} \right |_{t=0} \Vol(L_t) = - \int_L \langle H,V \rangle = - \int_L \langle i_H\omega,i_V\omega \rangle = - \int_L \phi |\delta i_H \omega|^2 , 
\end{equation}
which is zero because the variation is equivariant. This, together with arbitrariness of $\phi$, implies $\delta i_H\omega \equiv 0$ which finishes the proof.
\end{proof}

\section{Symplectic reduction and toric varieties}

\subsection{Symplectic reduction}

Let $(M,\omega)$ be a symplectic manifold, $\mathfrak{g} \cong {\mathbb R}^n$ be the Lie algebra of $T^n$, $\mathfrak{g}^* \cong {\mathbb R}^n$ be the dual vector space of $\mathfrak{g}$, and
$\psi : T^n \rightarrow Sympl(M,\omega)$ be a symplectic action. For each $X  \in \mathfrak{g}$, there is the one-parameter subgroup $\{exp(t X) | t \in {\mathbb R}\} \subseteq T^n$ and the corresponding $T^n$-invariant vector field $X^\#$ on~$M$.

\begin{defi}
An action $\psi$ is called Hamiltonian if there exists a moment map
\begin{equation}
\mu : M \rightarrow \mathfrak{g}^*
\end{equation}
satisfying the following property: for each unit basis vector $X_i \in {\mathbb R}^n \cong \mathfrak{g}$, the function $\mu_i$ is a Hamiltonian function for $X_i^\#$, i.e. $i_{X_i^\#}\omega=d \langle \mu (p), X_i \rangle = d \mu_i$
\end{defi}


\begin{exam}
Consider the standard symplectic structure $\omega= - i \sum_{k=1}^m dz_k \wedge d\overline{z_k}$ on ${\mathbb C}^m$. 
Consider the standard torus
\begin{equation}
T^m=\{ (z_1,\ldots,z_m ) \in {\mathbb C}^m | \quad |z_i|=1 \quad \text{for all} \quad 1 \leq i \leq m\} \subset {\mathbb C}^m.
\end{equation}
The coordinate-wise action $T^m:{\mathbb C}^m$ is Hamiltonian, and its moment map is
\begin{equation} 
\mu(z_1,\ldots,z_m)=(|z_1|^2,\ldots,|z_m|^2).
\end{equation}

\end{exam}

\begin{theo}[Symplectic reduction]
Let $(M,\omega,T^n,\mu)$ be a symplectic manifold with a Hamiltonian $T^n$-action. Assume that that the moment map $\mu$ is proper. Let $c$ be a regular value of $\mu$ and and $i:\mu^{-1} (c) \rightarrow M$ be the inclusion map. Assume that $T^n$ acts freely on $\mu^{-1}(c)$. Then:

\medskip

\noindent $\bullet$ the level set $\mu^{-1}(c)$ is a smooth compact $T^n$-invariant submanifold in $M$,

\medskip

\noindent $\bullet$ the orbit space $M_{red}=\mu^{-1}(c) / T^n$ is a manifold,

\medskip

\noindent $\bullet$ $\pi : \noindent \mu^{-1}(c) \rightarrow M_{red}$ is a principle $T^n$ bundle, and

\medskip

\noindent $\bullet$ there is a symplectic form $\omega_{red}$ on $M_{red}$ satisfying $i^* \omega = \pi ^* \omega$.
\end{theo}
\begin{rema}
If $M$ is compact then the moment map $\mu$ is proper.
\end{rema}
\begin{rema}
There is a more general version of symplectic reduction for Hamiltonian actions of compact connected Lie groups, which definition we did not give. The proof can be found in~\cite{MW}.
\end{rema}

\subsection{Moment-angle manifolds}
The details of the following construction can be found in~\cite[Construction 6.1.1]{BP}.
\begin{cons}
Consider a presentation of a convex polytope
\begin{equation}
P=P(A,\boldsymbol{b})=\{ \boldsymbol{x}  \in {\mathbb R}^n: \langle \boldsymbol{a}_i,\boldsymbol{x} \rangle + b_i \geq 0, \quad 1 \leq i \leq m   \}.
\end{equation}
Let $A$ be the $n\times m$-matrix with columns $\boldsymbol{a}_1,\dots,\boldsymbol{a}_m$.
The map
\begin{equation}
i_{A,\boldsymbol{b}}: {\mathbb R}^n \rightarrow {\mathbb R}^m, \quad i_{A,\boldsymbol{b}}(\boldsymbol{x})=A^t \boldsymbol{x} + \boldsymbol{b}=
(\langle \boldsymbol{a}_1,\boldsymbol{x} \rangle + b_1,\ldots,\langle \boldsymbol{a}_m,\boldsymbol{x} \rangle + b_m )^t 
\end{equation}
embeds $P$ into  ${\mathbb R}^m_{\geq}$. Define the space $Z_{A,\boldsymbol{b}}$ from the commutative diagram
\begin{equation}
\begin{CD}
Z_{A,\boldsymbol{b}}  @>i_{Z}>>                  {\mathbb C}^m \\
@VVV                                              @VV\mu V      \\
P               @>i_{A,\boldsymbol{b}}>>         {\mathbb R}^m_{\geq}
\end{CD}
\end{equation}
where $\mu(z_1,\ldots,z_m)=(|z_1|^2,\ldots,|z_m|^2)$ is the moment map of the coordinate-wise action of $T^m$ on ${\mathbb C}^m$.

Let $\Gamma$ be the matrix whose rows form a basis of the space $\{\boldsymbol{y} \in {\mathbb R}^m : \boldsymbol{y} A^t =0\}$. Then $\Gamma$
is an $(m-n)\times m$ matrix with maximal rank $m-n$ satisfying $\Gamma A^t=0$. The set of columns $(\boldsymbol{\gamma}_1,\ldots,\boldsymbol{\gamma}_m)$ of $\Gamma$ is called a Gale dual of the set $(\boldsymbol{a}_1,\ldots,\boldsymbol{a}_m)$ of columns of $A$. Now  we can write
\begin{equation}
i_{A,\boldsymbol{b}}({\mathbb R}^n)= \{ \boldsymbol{x} \in {\mathbb R}^m | \Gamma \boldsymbol{x}= \Gamma \boldsymbol{b} \}.
\end{equation}
Then the image of $Z_{A,\boldsymbol{b}}$ in ${\mathbb C}^m$ by the map $i_Z$ can be expressed as an intersection of Hermitian quadrics:
\begin{equation}
Z_\Gamma=i_Z (Z_{A,\boldsymbol{b}} )=\{(z_1,\dots,z_m) \in {\mathbb C}^m : \sum_{k=1}^m \gamma_{jk}|z_k|^2=\sum_{k=1}^m \gamma_{jk}b_k,  \quad 1\leq j \leq m-n \}.
\end{equation}
The space $Z(A,\boldsymbol{b})$ is a smooth manifold if and only if the polytope $P$ is simple (i.e. all its intersecting faces are in general position, which gives the condition of transversal intersection of quadrics), see~\cite[Theorem 6.1.3]{BP}. In this case $Z(A,\boldsymbol{b})$ is called a {\itshape moment-angle manifold}.
\end{cons}
Further we will always assume that $P$ is a simple polytope and use the notation $Z_\Gamma$ as well as $Z_P$ for the manifold $Z_{A,\boldsymbol{b}}$. Notice that the moment-angle manifold is compact since $P$ is bounded.

Real moment-angle manifolds are defined similarly from the commutative diagram
\begin{equation}
\begin{CD}
R_{A,\boldsymbol{b}}  @>i_{R}>>                  {\mathbb R}^m \\
@VVV                                              @VV\mu V      \\
P               @>i_{A,\boldsymbol{b}}>>         {\mathbb R}^m_{\geq}
\end{CD}
\end{equation}
where $\mu(y_1,\ldots,y_m)=(|y_1|^2,\ldots,|y_m|^2)$ is the restriction of the moment map. The manifold $R_P = R(A,\boldsymbol{b})$ can be written as an intersection of real quadrics:
\begin{equation}
R_{\Gamma}=\{ u=(u_1,\dots,u_m) \in {\mathbb R}^m : \sum_{k=1}^m \gamma_{jk} u_{k} ^2 = \sum_{k=1}^m \gamma_{jk}b_k, \quad   1 \le j \le m-n \}.
\end{equation}
There is an obvious inclusion  $R_{A,\boldsymbol{b}}  \hookrightarrow  Z_{A,\boldsymbol{b}}$.

The action $T^m : {\mathbb C}^m $ restricts to an action $T^m : Z_\Gamma$ and similarly the action ${\mathbb Z}_2^m : {\mathbb R}^m$ restricts to an action ${\mathbb Z}_2^m : R_\Gamma$. Both orbit spaces are identified with the initial polytope: $Z_\Gamma/T^m = R_\Gamma / {\mathbb Z}_2^m \cong P $.

\begin{exam}
Let $P$ be a triangle defined by equations
\begin{equation}
\left \{
\begin{aligned}
&x_1 \geq 0,\\
&x_2 \geq 0,\\
&1-x_1-x_2 \geq 0.\\
\end{aligned}
\right.
\end{equation}

Then the map $i_{A,\boldsymbol{b}}$ is given by the formula $i_{A,\boldsymbol{b}}(x_1,x_2)=(x_1,x_2,1-x_1-x_2)$ and $i_{A,\boldsymbol{b}}(P)={\mathbb R}^3_\ge \cap \{y_1+y_2+y_3=1\}$. Thus
\begin{equation}
\begin{aligned}
&Z_\Gamma=\{(z_1,z_2,z_3)\in {\mathbb C}^3||z_1|^2+|z_2|^2+|z_3|^2=1 \} \cong S^5,\\
&R_\Gamma=\{(y_1,y_2,y_3)\in {\mathbb R}^3|y_1^2+y_2^2+y_3^2=1 \} \cong S^2.\\
\end{aligned}
\end{equation}

\end{exam}

\subsection{Symplectic toric varieties}
We endow ${\mathbb C}^m$ with a standard symplectic structure $\omega= - i \sum_{k=1}^m dz_k \wedge d\overline{z_k}$. Recall that the action $T^m \colon {\mathbb C}^m$ is Hamiltonian and its moment map is given by the formula $ \mu(z_1,\ldots,z_m)=(|z_1|^2,\ldots,|z_m|^2) $.

Assume that the vectors $\boldsymbol{a_1},\ldots,\boldsymbol{a_m}$ span a lattice $N={\mathbb Z}\langle \boldsymbol{a_1},\ldots,\boldsymbol{a_m} \rangle \subset {\mathbb R}^n$. It is easily seen that then the columns $\boldsymbol{\gamma}_1,\ldots,\boldsymbol{\gamma}_m$ of matrix $\Gamma$ span a lattice $L={\mathbb Z}\langle \boldsymbol{\gamma_1},\ldots,\boldsymbol{\gamma_m} \rangle \subset {\mathbb R}^{m-n}$. Because ${\mathbb R} \langle  \boldsymbol{a_1},\ldots,\boldsymbol{a_m} \rangle =   {\mathbb R} ^n$ we have $N  \cong {\mathbb Z}^n$ and $L \cong  {\mathbb Z}^{m-n}$. Consider the following subgroup of $T^m$:
\begin{equation}
T_\Gamma={\mathbb R} ^{m-n} / L^*= \{(e^{2\pi i (\gamma_1,\phi)},\dots,e^{2\pi i (\gamma_m,\phi)}) \in {\mathbb T}^m \} \cong {\mathbb T}^{m-n},
\end{equation}
where $\phi \in {\mathbb R}^{m-n} $ and $L^* = \{ \lambda^* \in {\mathbb R}^{m-n} \colon (\lambda^*,\lambda) \in  {\mathbb Z} \quad \text{for all} \quad \lambda \in L\}$
is the dual lattice. The real analogue of $T_\Gamma$ is the discrete group $D_\Gamma = \frac{1}{2} L^* /L^* \cong ({\mathbb Z}/2)^{m-n}$, it canonically embeds as a subgroup in $T_{\Gamma}$. The restricted action of $T_\Gamma \subset T^m$ on ${\mathbb C}^m$ is also Hamiltonian, and the corresponding moment map is the composition
\begin{equation} 
\mu_\Gamma \colon {\mathbb C}^m \rightarrow {\mathbb R}^m \rightarrow \mathfrak{t}^{*}_\Gamma,
\end{equation}
where $ {\mathbb R}^m \rightarrow \mathfrak{t}^*_\Gamma$ is the map of the dual Lie algebras corresponding to the inclusion $T_\Gamma \hookrightarrow T^m$.
The map $ {\mathbb R}^m \rightarrow \mathfrak{t}^*_\Gamma$  sends the $i$-th basis vector $e_i \in {\mathbb R}^m$ to $\gamma_i \in \mathfrak{t}^*_\Gamma \cong {\mathbb R}^{m-n}$. Hence $\mu_\Gamma$ is obtained by composing the standard moment map $\mu$ with $\Gamma$. That is,
\begin{equation}
\mu_\Gamma(z_1,\ldots,z_m)= \Bigl( \sum_{k=1}^m \gamma_{1k}|z_k|^2,\ldots,\sum_{k=1}^m \gamma_{(m-n) k}|z_k|^2 \Bigr)
\end{equation}

The level set $\mu_\Gamma^{-1}(\Gamma \boldsymbol{b})$ is exactly the moment-angle manifold $Z_\Gamma$. From now on we assume that $T_\Gamma$ acts freely on $Z_\Gamma$, which is equivalent to the property of $P$ being Delzant (i.e. $P$ is simple and for every vertex the corresponding normal vectors of adjacent facets span the whole lattice: ${\mathbb Z} \langle \boldsymbol{a}_{i_1},\ldots,\boldsymbol{a}_{i_n} \rangle = N$).
The following theorem with proof can be found in~\cite[Theorem 6.3.1]{BP}.
\begin{theo}
Let $P=P(A,\boldsymbol{b})$ be a Delzant polytope, $\Gamma=(\boldsymbol{\gamma}_1,...\boldsymbol{\gamma}_m)$ the corresponding Gale dual configuration of vectors in ${\mathbb R}^{m-n}$, which defines the moment-angle manifold $Z_P=Z_\Gamma=Z_{A,\boldsymbol{b}}$. Then:

\medskip

\noindent $\bullet$ $\Gamma \boldsymbol{b}$ is a regular value of the proper moment map $\mu_\Gamma \colon{\mathbb C}^m \rightarrow \mathfrak{t}^*_\Gamma \cong {\mathbb R}^{m-n}$,

\medskip

\noindent $\bullet$ $Z_P$ is the regular level set $\mu_\Gamma^{-1}(\Gamma \boldsymbol{b})$,

\medskip

\noindent $\bullet$ the action $T_\Gamma$ on $Z_P$ is free.
\end{theo}
\begin{rema}
The conditions for applying symplectic reduction are as follows: the moment map is proper, the action is free and the level set is regular. These conditions are equivalent to the properties of $Z_\Gamma$ being bounded, $P$ being Delzant and $Z_\Gamma$ being a non-empty manifold, respectively.
\end{rema}

By applying symplectic reduction, we get a manifold $V_P=Z_P/T_\Gamma$. It is canonically isomorphic to a toric manifold $V_{\Sigma_P}$, which corresponds to the normal fan $\Sigma_P$ of the polytope $P$ (see~\cite[Theorem 5.5.4]{BP}). Manifolds $V_{\Sigma_P}$ obtained in this way are called {\itshape symplectic toric manifolds}.

A toric manifold $V_{\Sigma_P}$ is a projective algebraic variety. The reduced symplectic form $\omega_{red}$ and metric induced by the Riemannian submersion  $Z_P \rightarrow V_P$ are equivalent to the symplectic form induced by the projective embedding and the metric coming from the algebraic structure. Let us also notice the beautiful fact that the set of Delzant polytopes is in bijective correspondence with the set of equivariant symplectomorphism classes of symplectic toric manifolds~\cite{Delz}.

The important fact is that $R_P$ projects on real toric manifold $U_P$, which is a set of real points of complex toric manifold. The fiber is $D_\Gamma$. The main results of this paragraph can be illustrated by the following diagram~\eqref{diag1}, where fibers of the projections $\pi$ and $r$ are $T_\Gamma$ and $D_\Gamma$, and all dimensions are real except of ${\mathbb C}^{m}$.

\begin{equation}\label{diag1}
\begin{tikzcd}
R_P^n \arrow[hookrightarrow]{r} \arrow[rightarrow]{d}{r}
  & Z_P^{m+n} \arrow[hookrightarrow]{r}{i}\arrow[rightarrow]{d}{\pi}
  & {\mathbb C}^{m} \\
U_P^n \arrow[hookrightarrow]{r}
  & V_P^{2n}
\end{tikzcd}
\end{equation}

\section{Minimal submanifolds in moment-angle manifolds}

We fix our notation as on the diagram~\eqref{diag1}.

\begin{prop}\label{prop1}
The real toric manifold $U_P$ is a minimal submanifold of $V_P$.
\end{prop}
\begin{proof}
The real toric manifold $U_P$ is the fixed point set under isometric involution $\sigma \colon V_P$ induced by the complex conjugation on $Z_P \subset {\mathbb C}^m$. It is easily seen that $U_P$ is {\itshape totally geodesic} in $V_P$, i.e. all geodesics in  $U_P$ are geodesics in $V_P$. Therefore $U_P$ is minimal in $V_P$, because being totally geodesic means that the second quadratic form is zero, and minimality means that its trace is zero, i.e. mean curvature vector field vanishes everywhere (see~\cite[Chapter 1]{L} for the details).
\end{proof}

Define the function $\Vo \colon V_P \rightarrow {\mathbb R}$ as a volume of an orbit: 
\begin{equation}
\Vo(x)=\Vol(\pi^{-1}(x)).
\end{equation}
\begin{prop}\label{prop2}
The real toric manifold $U_P$ is minimal in $V_P$ with respect to the metric $\tilde{g}=Vo^{2/n} g$.
\end{prop}
\begin{proof}
The key observation is that the function $\Vo$ (volume of the preimage) is invariant under the involution $\sigma \colon V_P$ induced by the complex conjugation. This is because conjugation is isometry on $\mathbb{C}^m$. In fact the volume of an orbit depends only on absolute values of the coordinates of points in the orbit. Thus we have
\begin{align}
\begin{split}
& \Vo(x)=\Vol(\pi^{-1}(x))=\Vol(T_\Gamma(z_1,\ldots ,z_m))=\Vol(|z_1|,\ldots,|z_m|)=\\
&=\Vol(|\overline{z_1}|,\ldots,|\overline{z_m}|)=\Vol(T_\Gamma(\overline{z_1},\ldots ,\overline{z_m}))=\Vol(\pi^{-1}(\sigma x))=\Vo(\sigma x).
\end{split}
\end{align}
Hence the involution $\sigma$ is isometric not only with respect to the metric $g$, but also with respect to the metric $\tilde{g}=Vo^{2/n} g$. This means that the argument from Proposition~\ref{prop1} works also in the case of metric $\tilde{g}$.
\end{proof}

We will also need the following generalization of Noether's theorem:
\begin{theo}\label{theo4.3}
Let $(M,\omega,T,\mu)$ be a symplectic manifold with a Hamiltonian torus action. Let $X$ be a Hamiltonian $T$-invariant vector field. Then the moment map $\mu$ is constant along $X$.
\end{theo}
\begin{proof}
The proof is a sequence of equalities:
\begin{equation}
X(\mu_i) =i_X d \mu_i=  i_X i_{X^{\#}_i} \omega = - i_{X^\#_i} i_X \omega = - i_{X^\#_i}  d f = - X^\#_i(f) = 0,
\end{equation}
because the vector field $X$ is $T$-invariant and therefore the 1-form ${d}f=i_X \omega$ is $T$-invariant.
\end{proof}

Now define a submanifold 
\begin{equation}
N= \pi ^{-1}(U_P) \cong R_P \times_{D_\Gamma} T_\Gamma  
\end{equation}
of the moment-angle manifold $Z_P \subset {\mathbb C}^m $.
\begin{theo}\label{theoA}
The submanifold $N$ is minimal in $Z_P$.
\end{theo}
\begin{proof}
The action $T_\Gamma \colon Z_P$ is free and $Z_P/T_\Gamma=V_P$. Hence there is a bijection:
\begin{equation}
\left \{
\begin{aligned}
& T_{\Gamma}  \text {-invariant horizontal} \\
&\text{vector fields on} \quad  Z_P \\
\end{aligned}
\right \}
\longleftrightarrow
\left \{
\begin{aligned}
&\text{vector fields} \\
&\text{on} \quad V_P \\
\end{aligned}
\right \}
\end{equation}
Assume that $i_t$ is a $T_\Gamma$-invariant variation of the natural embedding $i \colon N\hookrightarrow Z_P$. By the definition of $\tilde{g}$,
\begin{equation}
\Vol(N_t,g)=\Vol(\pi(N_t),\tilde{g}).
\end{equation}
Comparing $T_\Gamma$-invariant variations of $N$ in $Z_P$ and variations of $\pi(N)= U_P$ in $V_P$, we obtain by minimality of $U_P$ w.r.t. metric $\tilde{g}$ (Proposition~\ref{prop2}) that the volume of $N$ is stationary with respect to all $T_\Gamma$-invariant variations. This together with Theorem~\ref{theo2.3} implies the minimality of~$N$.
\end{proof}

\begin{theo}\label{theoB}
The manifold $N$ is an H-minimal Lagrangian submanifold of ${\mathbb C}^m $.
\end{theo}
\begin{proof}
We have a sequence of embeddings: $N \hookrightarrow Z_P \hookrightarrow {\mathbb C}^m$.

We first prove that $N$ is a Lagrangian submanifold in ${\mathbb C}^m$. For any $x \in N$, there is a decomposition of the tangent space $\mathcal T_x N=\mathcal T_{R_\Gamma} \oplus \mathcal T_{T_\Gamma}$ into the sum of the tangent space of the real moment-angle manifold and the tangent subspace along the torus action. The canonical symplectic form $\omega= - i \sum_{k=1}^m dz_k \wedge d\overline{z_k}$ on ${\mathbb C}^m $ vanishes on the tangent subspace $\mathcal T_{R_\Gamma}$, because there are only real vectors (tangent to $R_\Gamma$) in $\boldsymbol{z} \cdot \mathcal T_{R_\Gamma}$ for some $\boldsymbol{z} \in T_\Gamma$. Now assume that $X_i \in \mathcal T_{T_\Gamma}$ and $Y \in  \mathcal T_x N$. Then
\begin{equation} 
\omega(X_i,Y)=i_{X_i} \omega (Y)=d \mu_i (Y)= Y(\mu_i) =0, 
\end{equation}
since $Y \in  \mathcal T_x N$ and therefore $\mu$ is constant along $Y$ (because $N\subset Z_\Gamma= \mu^{-1} (\Gamma \boldsymbol{b})$). This implies that $\omega$ vanishes on the whole tangent space $\mathcal T_x N=\mathcal T_{R_\Gamma} \oplus \mathcal T_{T_\Gamma}$, i.e. $N$ is Lagrangian.

Now let us prove the H-minimality. By Proposition 2.5, we can consider only $T_\Gamma$-invariant Hamiltonian variations. By Theorem~\ref{theo4.3}, any $T_\Gamma$-invariant Hamiltonian variation of $N\subset Z_\Gamma\subset {\mathbb C}^m$ belongs to $Z_\Gamma$. Finally by Theorem~\ref{theoA}, the volume of $N$ is stationary with respect to all $T_\Gamma$-invariant variations in $Z_\Gamma$.
\end{proof}

\begin{exam}[One quadric]
Let $m-n=1$, i.e. $Z_\Gamma$ is defined by one equation
\begin{equation}
\gamma_1 |z_1|^2 + \ldots + \gamma_m |z_m|^2=c.
\end{equation}
Compactness implies that all coefficients are positive. The action $T_\Gamma \colon Z_\Gamma$ is free (equivalently, the initial polytope is Delzant) if and only if the following condition is satisfied: for every point $\boldsymbol{z} \in Z_\Gamma$ there is an equality ${\mathbb Z} \langle \gamma_{i_1},\ldots ,\gamma_{i_k} \rangle ={\mathbb Z} \langle \gamma_{1},\ldots ,\gamma_{k} \rangle =L$,
where $z_{i_1},\ldots,z_{i_k}$ are the only non-zero coordinates of~$\boldsymbol{z}$ (see~\cite[Theorem 4.1]{MP}). Because in our case $Z_\Gamma$ contains points with only one non-zero coordinate, every $\gamma_i$ should generate the same lattice as the whole set $\gamma_1,\ldots,\gamma_m$. Therefore $\gamma_1=\ldots =\gamma_m$, and $Z_\Gamma$ is a sphere $S^{2m-1}$ of radius $\sqrt{a}=\sqrt{\frac c \gamma_1}$ defined by equation
\begin{equation}|z_1|^2 + \ldots + |z_m|^2=a.
\end{equation}
The manifold $R_\Gamma \subset Z_\Gamma$ is a sphere in real part of $Z_\Gamma$:
\begin{equation}
S^{m-1}=\{ (r_1,\ldots,r_m) \in {\mathbb C}^m |\quad r_i \in {\mathbb R} \quad 1 \leq i \leq m,\quad r_1^2 + \ldots +r_{m}^2=a \}.
\end{equation}
In order to get $N=R_\Gamma \times_{D_\Gamma} T_\Gamma $, one should <<spread>> the sphere $R_\Gamma \cong S^{m-1}$ by the action of the circle $T_\Gamma= \{(e^{2\pi i \phi)},\dots,e^{2\pi i \phi)}) \in {\mathbb C}^m \} \cong S^1$. In this way, depending on whether the involution changes the orientation on $S^{m-1}$, we get the manifold

\begin{equation}
\begin{aligned}
&N(m) \cong S^{m-1}\times S^1 && \text{for even} \quad m, \\
&N(m) \cong K^m && \text{for odd} \quad m,
\end{aligned}
\end{equation}
where $K^m$ is an $m$-dimensional Klein bottle.
The submanifold $N(m)$ is minimal in $S^{2m-1}$ and is H-minimal Lagrangian in ${\mathbb C}^m$.
\end{exam}

\begin{exam}[Two quadrics]
Let $m-n=2$, then $Z_\Gamma$ is defined by equations
\begin{equation}
\left \{
\begin{aligned}
&\gamma_{11} |z_1|^2 + \ldots + \gamma_{m1} |z_m|^2=c, \\
&\gamma_{12} |z_1|^2 + \ldots + \gamma_{m2} |z_m|^2=0,\\
\end{aligned}
\right.
\end{equation}
where $\gamma_{k1} > 0$, $c > 0$, $\gamma_{j2} > 0$, $\gamma_{i2} < 0$  for $1 \leq k \leq m$, $1 \leq j \leq p$, $p+1 \leq i \leq m$ (this is the canonical form of an intersection of two quadrics described in~\cite[Proposition 4.2]{MP}). The second equation defines a cone over the product of two ellipsoids of dimensions $2p-1$ and $2q-1$. By intersecting it with ellipsoid of dimension $2m-1$, defined by first equation, we obtain that $Z_\Gamma \cong S^{2p-1} \times S^{2q-1}$ and $R_\Gamma \cong S^{p-1}\times S^{q-1}$. Let us note that the corresponding polytope is combinatorially equivalent to the product of simplices $\triangle^{p-1}\times\triangle^{q-1}$.

In~\cite{MP} it is proved that in this case
\begin{equation} 
N_\Gamma = N_l (p,q)= R_\Gamma \times_{D_\Gamma} T_\Gamma \approx (S^{p-1}\times S^{q-1}) \times_{\mathbb{Z}_2\times \mathbb{Z}_2} S^1 \times S^1,
\end{equation}
where the actions on the right side are antipodal involutions, and on the left side they are given by

\begin{equation}
\begin{aligned}
&\varphi_1: (u_1,\dots ,u_m)\rightarrow(-u_1,\dots,-u_l,-u_{l+1},\dots,-u_p,u_{p+1},\dots,u_m),\\
&\varphi_2: (u_1,\dots ,u_m)\rightarrow(-u_1,\dots,-u_l,u_{l+1},\dots,u_p,-u_{p+1},\dots,-u_m).
\end{aligned}
\end{equation}

\noindent So the topological type of $N_\Gamma$ is defined by three numbers $p$,$q$ and $l$ where $p+q=m$ and $0 \leq l \leq p$. 

In general there is a fiber bundle
\begin{equation}
\begin{tikzcd}
R_\Gamma \times_{D_\Gamma} T_\Gamma = N_\Gamma \arrow[rightarrow]{r}{R_\Gamma}
  & T_\Gamma/D_\Gamma \cong T^{m-n},
\end{tikzcd}
\end{equation}
so in case of $m-n=2$ we get
\begin{equation}
\begin{tikzcd}[column sep=large]
N_l(p,q) \arrow[rightarrow]{r}{S^{p-1} \times S^{q-1}}
  & T^{2}.
\end{tikzcd}
\end{equation}
Notice that transformations of fibers in this bundle are described by $\varphi_1$ and $\varphi_2$. If $p$ is odd then $\varphi_1$ changes the orientation of $R_\Gamma = S^{p-1} \times S^{q-1}$. This implies that the fiber bundle is unorientable and hence is not trivial. Same argument works if $l+q$ is odd.

We also have $N_l (p,q) =N(p)\times_{{\mathbb Z}/2} (S^{q-1} \times S^1) $ because $\varphi_1$ does not act on $S^{q-1}$ and acts by antipodal involution on $S^{p-1}$
(where $N(p)$ is the manifold from the case of $m-n=1$). This also gives us another interesting fiber bundle
\begin{equation}
\begin{tikzcd}[column sep=large]
N_l(p,q) \arrow[rightarrow]{r}{N(p)}
  & N(q)
\end{tikzcd}
\end{equation}
which topology depends on $l$. If $p$ and $q$ are even, but $l$ is odd, then this bundle is non-trivial, because we know that $N_l(p,q) \neq R_\Gamma \times T^2 = S^{p-1} \times S^{q-1} \times T^2 = N(p) \times N(q)$. If numbers $p$,$q$ and $l$ are even then this bundle is trivial, because of the fact that antipodal involution of odd-dimensional sphere (or even-dimensional vector space) is isotopic to the identity by rotations in two-planes.

The manifold $N$ is a minimal submanifold of $Z_\Gamma \cong S^{2p-1} \times S^{2q-1}$ and an H-minimal Lagrangian submanifold of ${\mathbb C}^m$. If $p=q=2$ and $l=1$ we get a minimal embedding $N_1(2,2)\hookrightarrow Z_\Gamma \cong S^{3} \times S^{3}$, where $N_1(2,2)\rightarrow T^2$ is a nontrivial bundle with fiber $T^2$. In the same time there is a trivial bundle $T^4=T^2\times T^2=N_0(2,2)$ which minimally embeds into $Z_\Gamma \cong S^{3} \times S^{3}$. The latter embedding is a product of two minimal embeddings $T^2 \hookrightarrow S^3$ constructed from one quadric from the previous example.
\end{exam}

\section{H-minimal Lagrangian submanifolds in toric manifolds}

Here we review the construction of Mironov and Panov~\cite{MP2} and give a more
explicit proof of their main theorem, which is independent of the results
of~\cite{D}.

Consider two sets of quadrics $Z_\Gamma$ and $Z_\Delta$:
\begin{align}
&Z_\Gamma= \{\boldsymbol{z}=(z_1,\dots,z_m) \in {\mathbb C}^m \colon \sum_{k=1}^m \gamma_{jk}|z_k|^2= c_j,  \quad 1\leq j \leq m-n  \}, \\
&Z_\Delta=  \{\boldsymbol{z}=(z_1,\dots,z_m) \in {\mathbb C}^m \colon \sum_{k=1}^m \delta_{jk}|z_k|^2= d_j,  \quad 1\leq j \leq m-l \},
\end{align}
and denote $Z_\Gamma \cap Z_\Delta$ as $Z_{\Gamma \Delta}$, quadric corresponding to the matrix $\Gamma \Delta$ obtained by placing the matrix $\Gamma$ right on top of the matrix $\Delta$. We assume that  $Z_\Gamma$, $Z_\Delta$ and $Z_{\Gamma \Delta}=Z_\Gamma \cap Z_\Delta$ are non-degenerate rational intersections of quadrics. For instance for $Z_\Gamma$ this means that the following three conditions are satisfied:

\begin{itemize}

\item[(a)] $\boldsymbol{c} \in R_{\ge} \langle \gamma_1,\ldots,\gamma_m \rangle$,

\item[(b)] if $\boldsymbol{c} \in R_{\ge} \langle \gamma_{i_1},\ldots,\gamma_{i_k} \rangle$, then $k\geq m-n$,

\item[(c)] $\gamma_1,\ldots,\gamma_m$ generate a lattice $L$ of maximal rank in ${\mathbb R}^{m-n}$.
\end{itemize}

\noindent We also assume that the polytopes associated with the intersections of quadrics  $Z_\Gamma$, $Z_\Delta$ and $Z_{\Gamma \Delta}$ are Delzant (for the construction of the polytope associated with an intersection of quadrics, see~\cite{MP}).

The groups $T_\Delta$, $T_{\Gamma \Delta}$, $D_\Delta$, and $D_{\Gamma \Delta}$ are defined similarly as the groups $T_\Gamma$ and $D_\Gamma$. Notice that $n+l$ should be $\ge m$ and $T_\Gamma^{m-n}$ is transverse to $T_\Delta^{m-l}$ inside $T_{\Gamma \Delta}^{2m-n-l} \subset T^m$, and all these three tori act freely on $Z_{\Gamma \Delta}$.

The idea is to use the first set of quadrics to produce a toric manifold $V_\Gamma$ via symplectic reduction as in diagram \eqref{diag1}, and
then use the second set of quadrics to define an H-minimal Lagrangian submanifold in $V_\Gamma$. We have the toric manifold $V_\Gamma=Z_\Gamma/T_\Gamma$ with a Hamiltonian torus action $T_\Delta \colon V_\Gamma$. The moment map $\mu_\Delta \colon V_\Gamma \rightarrow {\mathbb R}^{m-l}$ for the action $T_\Delta : V_\Gamma$ is given by the formula  
\begin{equation}
\mu_\Delta(x)=\Delta \cdot \mu ( \pi^{-1}(x))= \Delta \cdot (|z_1|^2,\ldots,|z_m|^2)^t
\end{equation}
where $\pi^{-1}(x)=(z_1,\ldots,z_m)$ is any preimage of $x$ by the map $\pi \colon Z_\Gamma \rightarrow V_\Gamma$. Now we can apply symplectic reduction by the action $T_\Delta \colon V_\Gamma$. Everything we need can be seen on the following diagram:

\begin{equation}
\begin{tikzcd}
  & Z_{\Gamma \Delta}^{n+l} \arrow[hookrightarrow]{r} \arrow[rightarrow]{d}
  & Z_{\Gamma}^{m+n} \arrow[hookrightarrow]{r} \arrow[rightarrow]{d}{\pi}
  & {\mathbb C}^{m} \arrow[dotted]{dl}{\text{symplectic reduction by the action of $T_\Gamma$}}\\
(R_{\Gamma \Delta}/D_\Gamma)^{n+l-m} \arrow[hookrightarrow]{r} \arrow[rightarrow]{d}{\tilde{r}}
  & (Z_{\Gamma \Delta}/T_\Gamma)^{2n+l-m} \arrow[hookrightarrow]{r}{i}\arrow[rightarrow]{d}{\tilde{\pi}}
  & V_\Gamma^{2n} \arrow[dotted]{dl}{\text{symplectic reduction by the action of $T_\Delta$}}\\
U_{\Gamma \Delta}^{n+l-m} \arrow[hookrightarrow]{r}
  & V_{\Gamma \Delta}^{2(n+l-m)}
\end{tikzcd}
\end{equation}
where the fibers of the projections $\tilde{\pi}$ and $\tilde{r}$ are $T_\Delta$ and $D_\Delta$, respectively.

Consider the $n$-dimensional submanifold 
\begin{equation}
\tilde{N} = \tilde{\pi} ^{-1}(U_{\Gamma \Delta}) \cong (R_{\Gamma \Delta}/D_\Gamma) \times_{D_\Delta} T_\Delta
\end{equation} 
of the manifold $ Z_{\Gamma \Delta}/T_\Gamma  \subset V_\Gamma $. Similarly to Theorems~\ref{theoA} and~\ref{theoB}, we get:
\begin{theo}
The manifold $\tilde{N}$ is a minimal submanifold of the manifold $ Z_{\Gamma \Delta}/T_\Gamma $ and an H-minimal Lagrangian submanifold of toric manifold $V_\Gamma$.
\end{theo}
\begin{proof}
We have a sequence of embeddings: $ \tilde{N} \hookrightarrow Z_{\Gamma \Delta} / T_\Gamma \hookrightarrow V_\Gamma$. As before the manifold $\tilde{N}$ is a Lagrangian submanifold in $V_\Gamma$ because the tangent space decomposes $\mathcal T_x \tilde{N}=\mathcal T_{R_{\Gamma \Delta}/ D_\Gamma} \oplus \mathcal T_{T_\Delta}$ into the real subspace and the torus action subspace. We only have to note that symplectic form $\omega_{red}$ on $V_\Gamma$ is equal to the canonical Kaehler form, which is zero on the real part $R_\Gamma$, and thus will be zero on $R_{\Gamma \Delta}/D_\Gamma$.

Let us prove H-minimality. We can consider only $T_\Delta$-invariant Hamiltonian variations. By Noether's theorem all these variations of $\tilde{N}$ will be lying inside $Z_{\Gamma \Delta} / T_\Gamma$. It is left to show that $\tilde{N}$ is minimal in $Z_{\Gamma \Delta} / T_\Gamma$. As before in Theorem~\ref{theoA}, this follows from $U_{\Gamma \Delta}$ being minimal in $V_{\Gamma \Delta}$ with the appropriately corrected metric $\tilde{ \tilde{g}} =( \Vo ^{ \frac {2} {n+l-m}} )g$ on $V_{\Gamma \Delta}$, where $g$ is the standart metric on $V_\Gamma$, and function $\Vo(x)=\Vol(\tilde{\pi}^{-1}(x) )$ is the volume of the orbit. For this fact the reason is that conjugation is isometry on $V_\Gamma$ (just as on $\mathbb{C}^m$), and thus preserves the volume of the orbits $\tilde{\pi}^{-1}(x)$ . 
\end{proof}
\begin{exam}
(a). If $m-n=0$, then the set of quadrics defining $Z_\Gamma$ is void, so $Z_\Gamma=V_\Gamma={\mathbb C}^m$ and we obtain the original construction of submanifolds $N$ which are minimal in $Z_\Delta$ and H-minimal Lagrangian in ${\mathbb C}^m$.

(b). If $m-l=0$, then the set of quadrics defining $Z_{\Delta}$ is void, so $\tilde{N}$ is a real toric manifold $U_\Gamma$, which is minimal (even totally geodesic) in $V_\Gamma$.

(c). If $m-n=1$, then $Z_\Gamma \cong S^{2m-1}$, so we get H-minimal Lagrangian submanifolds $\tilde{N}$ of the projective space $V_\Gamma=Z_\Gamma / S^1={\mathbb C}P^{m-1}$, which also embed minimally into submanifold $Z_{\Gamma \Delta} / T_\Gamma = (S^{2m-1} \cap Z_\Delta )/ S^1$. This family contains H-minimal Lagrangian tori in $\mathbb{C}P^2$ and $\mathbb{C}P^3$ constructed in~\cite{Mi1},~\cite{MZ}.
\end{exam}

\section*{Acknowledgments}

This work is a result of the author's senior thesis at Lomonosov Moscow State University. The author is  grateful to his advisor T.E.Panov for suggesting the problem and continuous support during the research.

\address{Lomonosov Moscow State University}

\email{artofkot@gmail.com}

\end{document}